\documentclass[a4paper,UKenglish,cleveref, autoref, thm-restate]{lipics-v2021}

\nolinenumbers

\usepackage{url}
\usepackage{amsthm}
\usepackage{booktabs}
\usepackage{algorithm}
\usepackage{algorithmic}
\usepackage{xspace}
\usepackage{tikz}
\usepackage{amsmath,amssymb}
\usepackage{amsfonts}
\usepackage{float}
\usepackage{caption}
\usepackage{enumitem}
\usepackage{graphicx}

\usepackage{todonotes}
\usepackage{bm}

\usepackage{hyperref}
\usepackage{xcolor}
\hypersetup{
    colorlinks=true,
    linkcolor=blue,
    urlcolor=blue,
    citecolor=blue
}

\usepackage{bm}

\title{Parameterized complexity of $r$-Hop, $r$-Step, and $r$-Hop Roman Domination} 
\titlerunning{}

\author{Sandip Das}{Indian Statistical Institute, Kolkata, India}{sandip.das69@gmail.com}{https://orcid.org/0000-0002-1825-0097}{}

\author{Sweta Das}{Indian Statistical Institute, Kolkata, India}{sweta.twin@gmail.com}{https://orcid.org/0000-0002-1825-0097}{}

\author{Sk Samim Islam}{Indian Statistical Institute, Kolkata, India}{samimislam08@gmail.com}{[orcid]}{}

\authorrunning{S. Das, S. Das and S.S.Islam} 

\Copyright{Sandip Das \and Sweta Das  \and Sk Samim Islam}



\category{} 

\begin{document}

\maketitle

\section{Abstract}
The \textsc{Dominating Set} problem is a classical and extensively studied topic in graph theory and theoretical computer science. In this paper, we examine the algorithmic complexity of several well-known exact-distance variants of domination, namely \textsc{$r$-Step Domination}, \textsc{$r$-Hop Domination}, and \textsc{$r$-Hop Roman Domination}.

Let $G$ be a graph and let $r \geq 2$ be an integer. A set $S \subseteq V(G)$ is an \emph{$r$-hop dominating set} if every vertex in $V(G)\setminus S$ is at distance exactly $r$ from some vertex of $S$. Similarly, $S$ is an \emph{$r$-step dominating set} if every vertex of $G$ lies at distance exactly $r$ from at least one vertex of $S$. An \emph{$r$-hop Roman dominating function} on $G$ is a function $f \colon V(G)\to\{0,1,2\}$ such that for every vertex $v$ with $f(v)=0$, there exists a vertex $u$ at distance exactly $r$ from $v$ with $f(u)=2$. The \emph{weight} of $f$ is defined as $f(V)=\sum_{v\in V(G)} f(v)$. The \textsc{$r$-Hop Domination} (respectively, \textsc{$r$-Step Domination}) problem asks whether $G$ admits an $r$-hop dominating set (respectively, $r$-step dominating set) of size at most $k$, while the \textsc{$r$-Hop Roman Domination} problem asks whether $G$ admits an $r$-hop Roman dominating function of weight at most $k$.

It is known that for every $r\ge 2$, the problems \textsc{$r$-Step Domination}, \textsc{$r$-Hop Domination}, and \textsc{$r$-Hop Roman Domination} are \textsc{NP}-complete. We study their parameterized complexity. First we prove that for all $r\ge 2$, \textsc{$r$-Hop Roman Domination} is \textsc{W[2]}-complete. Furthermore, for every $r\ge 2$, \textsc{$r$-Step Domination} and \textsc{$r$-Hop Domination} remain \textsc{W[2]}-hard even when restricted to bipartite graphs and chordal graphs. Our reductions also imply that, unless the Exponential Time Hypothesis (ETH) fails, none of these problems admits an algorithm running in time $2^{o(n+m)}$ on graphs with $n$ vertices and $m$ edges.

\section{Introduction}
Domination in graphs is a fundamental concept in combinatorial optimization with numerous applications in facility location, network design, and distributed computing. Classical domination requires every vertex not in the dominating set to have a neighbor in the set. However, many practical scenarios demand more sophisticated control requirements, leading to various generalizations of domination. In communication networks, for instance, nodes may need to be monitored or controlled by other nodes at specific distances to ensure efficient routing or fault tolerance. This need has motivated the study of distance-based domination variants, including $r$-step, $r$-hop, and $r$-hop roman domination.

In 2000, ReVelle and Rosing~\cite{revelle2000defendens} introduced the 
\textsc{Roman Domination} problem, motivated by a classical military strategy problem. 
The problem was later studied by Stewart~\cite{stewart2012soldier}.

A \emph{Roman dominating function (RDF)} on a graph $G=(V,E)$ is a function 
$f:V\rightarrow\{0,1,2\}$ such that for every vertex $v\in V$ with $f(v)=0$, 
there exists a neighbor $u\in N(v)$ with $f(u)=2$. 
The \emph{weight} of $f$ is $f(V)=\sum_{v\in V} f(v)$, and the 
\emph{Roman domination number} $\gamma_R(G)$ is the minimum possible weight of an RDF on $G$. 
For an RDF $f$, the vertex set $V$ is partitioned into 
$V_0=\{v\in V \mid f(v)=0\}$, 
$V_1=\{v\in V \mid f(v)=1\}$, and 
$V_2=\{v\in V \mid f(v)=2\}$, 
so that $f$ can be represented as the ordered triple $(V_0,V_1,V_2)$~\cite{cockayne2004roman}. The \textsc{Roman Domination} problem asks whether $G$ admits an RDF of weight at most $k$.

Nascimento and Sampaio~\cite{nascimento2015roman} showed that 
\textsc{Roman Domination} is \textsc{NP}-hard even for subgraphs of grids and 
\textsc{APX}-hard for bipartite graphs with maximum degree four. 
They also proved fixed-parameter tractability for graph classes with bounded local treewidth, 
including bounded-degree and bounded-genus graphs (e.g., planar or toroidal graphs). 
Fernau~\cite{fernau2008roman} proved that \textsc{Roman Domination} is \textsc{W[2]}-complete.

Asemian et al.~\cite{asemian2025np} introduced the distance-based variant called the 
\emph{$r$-hop Roman dominating function} ($r$HRDF). 
For an integer $r\ge 2$, an $r$HRDF on $G=(V,E)$ is a function 
$f:V\rightarrow\{0,1,2\}$ such that for every $v\in V$ with $f(v)=0$, 
there exists $u\in V$ with $f(u)=2$ and $d_G(u,v)=r$. 
As before, $V$ is partitioned into $V_0,V_1,$ and $V_2$ according to the labels of $f$. The \textsc{$r$-Hop Roman Domination} problem asks whether $G$ admits an $r$HRDF of weight at most $k$.

They proved that for $r\ge 2$, the \textsc{$r$-Hop Roman Domination} problem is 
\textsc{NP}-complete even on planar bipartite and planar chordal graphs. 
However, their reduction does not establish parameterized hardness.

In this paper, we use a different reduction technique to determine the parameterized 
complexity of \textsc{$r$-Hop Roman Domination}. 
We present a parameterized reduction from the classical \textsc{Domination} problem, 
which yields the following result.

\begin{theorem}\label{thm:combined_r_roman_hop}
Let $r \geq 2$. The \textsc{$r$-Hop Roman Domination} problem satisfies the following:
\begin{enumerate}
    \item It is \textsc{W[2]}-complete.
    \item Unless ETH\footnote{The \textit{Exponential Time Hypothesis (ETH)} states that \textsc{3-SAT} cannot be solved in time $2^{o(n)}$, where $n$ is the number of variables in the input CNF formula.} fails, it does not admit a $2^{o(n+m)}$-time algorithm, where $n $ is the number of vertices and $m$ is the number of edges of the graph.
    \item For any $\epsilon > 0$, it cannot be approximated within a factor of $(1-\epsilon)\log n$, unless $\textsc{P} = \textsc{NP}$.
\end{enumerate}
\end{theorem}

Chartrand et al.~\cite{chartrand1995exact} introduced the notion of 
\emph{$r$-step domination}. For an integer $r \ge 2$, a vertex $u$ is said to be 
$r$-step dominated by a vertex $v$ if $d_G(u,v)=r$. 
A set $S \subseteq V(G)$ is an \emph{$r$-step dominating set} if every vertex of $G$ 
is $r$-step dominated by some vertex in $S$. 
The minimum cardinality of such a set is called the 
\emph{$r$-step domination number}, denoted $\gamma_{r\text{-step}}(G)$. 
The \textsc{$r$-Step Domination} problem asks whether a graph $G$ 
admits an $r$-step dominating set of size at most $k$.

This concept has been further studied in~\cite{caro2003some,dror2004note,hersh1999exact}. 
Henning et al.~\cite{henning20172} showed that \textsc{$2$-Step Domination} is 
\textsc{NP}-complete for planar bipartite and planar chordal graphs. 
Farhadi et al.~\cite{jalalvand2017complexity} established that for every $r \ge 2$, 
\textsc{$r$-Step Domination} is \textsc{NP}-complete even when restricted to 
planar bipartite and planar chordal graphs. 
However, their reductions do not imply parameterized hardness.

In this paper, we use a different reduction technique to determine the 
parameterized complexity of \textsc{$r$-Step Domination}. 
Our results are summarized below.

\begin{restatable}{theorem}{rstepdominationrestricted}
\label{thm:$r$-step domination restricted}

Let $r \geq 2$. The \textsc{$r$-Step Domination} problem satisfies the following:
\begin{enumerate}
    \item It is \textsc{W[2]}-complete, even when restricted to bipartite graphs and chordal graphs.
    \item Unless ETH fails, it does not admit a $2^{o(n+m)}$-time algorithm, where $n $ is the number of vertices and $m$ is the number of edges of the graph.
    \item For any $\epsilon > 0$, it cannot be approximated within a factor of $(1-\epsilon)\log n$, unless $\textsc{P} = \textsc{NP}$.
\end{enumerate}
\end{restatable}

A related notion, called \emph{hop domination}, was introduced by 
Ayyaswamy et al.~\cite{natarajan2015hop}. 
A set $S \subseteq V(G)$ is a \emph{hop dominating set (HDS)} 
if every vertex in $V(G)\setminus S$ is at distance exactly two from some vertex in $S$. 
The minimum size of such a set is the \emph{hop domination number}, denoted $\gamma_h(G)$. 
A natural generalization, called \emph{$r$-hop domination} for $r \ge 2$, 
was later studied in~\cite{jalalvand2017complexity}.

For a graph $G$ and an integer $r \ge 2$, a set $S \subseteq V(G)$ 
is an \emph{$r$-hop dominating set} ($r$HDS) if every vertex in $V(G)\setminus S$ 
is at distance exactly $r$ from some vertex in $S$. 
The minimum cardinality of such a set is the 
\emph{$r$-hop domination number}, denoted $\gamma_{rh}(G)$. 
The \textsc{$r$-Hop Domination} problem asks whether $G$ 
admits an $r$-hop dominating set of size at most $k$.

The case $r=2$ corresponds to the classical \textsc{Hop Domination} problem. 
Henning et al.~\cite{henning20172} showed that \textsc{Hop Domination} 
is \textsc{NP}-complete for planar bipartite and planar chordal graphs. 
They further proved that computing a minimum hop dominating set 
in an $n$-vertex graph cannot be approximated within a factor of 
$(1-\varepsilon)\log n$ for any $\varepsilon>0$, unless $\textsc{P}=\textsc{NP}$, 
and provided a polynomial-time approximation algorithm with a ratio 
$1+\log\bigl(\Delta(\Delta-1)+1\bigr)$, where $\Delta$ is the maximum degree of $G$~\cite{henning2020algorithm}. 
Moreover, the problem is \textsc{APX}-complete for bipartite graphs of maximum degree $3$.

Karthika et al.~\cite{karthika2025polynomial} gave polynomial-time algorithms 
for \textsc{Hop Domination} on interval graphs and biconvex bipartite graphs. 
They also studied its parameterized complexity, showing that the problem 
is \textsc{W[1]}-hard when parameterized by solution size, and later strengthened 
this result to \textsc{W[2]}-hardness~\cite{karthika2025hop}. 
Farhadi et al.~\cite{jalalvand2017complexity} proved that for every $r \ge 2$, 
\textsc{$r$-Hop Domination} is \textsc{NP}-complete even on planar bipartite 
and planar chordal graphs; however, their reduction does not imply parameterized hardness.

We study the parameterized complexity of \textsc{$r$-Hop Domination} 
and obtain the following result.

\begin{restatable}{theorem}{rhopdominationrestricted}
\label{thm:$r$-hop domination restricted}
Let $r \geq 2$. The \textsc{$r$-Hop Domination} problem satisfies the following:
\begin{enumerate}
    \item It is \textsc{W[2]}-complete, even when restricted to bipartite graphs and chordal graphs.
    \item Unless ETH fails, it does not admit a $2^{o(n+m)}$-time algorithm, where $n $ is the number of vertices and $m$ is the number of edges of the graph.
    \item For any $\epsilon > 0$, it cannot be approximated within a factor of $(1-\epsilon)\log n$, unless $\textsc{P} = \textsc{NP}$.
\end{enumerate}

\end{restatable}

\medskip
\noindent\textbf{Organisation:} In Section \ref{sec:prelim}, we recall some definitions and notations.  In Sections \ref{Roman_Hop_Domination}, \ref{sec:rStepDomination}, and Section \ref{sec:rHopDomination}, we studied the hardness results of \textsc{$r$-Hop Roman Domination}, \textsc{$r$-Step Domination}, and \textsc{$r$-Hop Domination} problems, respectively. We conclude in Section \ref{sec:conclusion}.

\section{Preliminaries}\label{sec:prelim}
Let $G$ be a graph with vertex set $V(G)$ and edge set $E(G)$.
The \textit{degree} of a vertex $u \in V(G)$ is the total number of edges, incident to $u$, denoted by $deg(u)$. We denote $d(u,v)$ as the length of a shortest path joining $u$ and $v$ in $G$.
The \textit{open neighborhood} of a vertex $u \in V(G)$, $N(u)$, is the set of all the vertices $v \in V(G)$ such that $(u,v)$ is an edge in $E(G)$.
The \textit{closed neighborhood} of a vertex $u$, $N[u] = N(u) \cup \{u\}$
The \textit{maximum degree} of a graph $G$, $\Delta(G) = max\{deg(v):v\in V(G)\}$. Given an undirected graph $G = (V(G), E(G))$, a subset of vertices $D\subseteq V$ is called a dominating set if for every vertex $u \in V(G)\setminus D$, there is a vertex $v\in D$  such that $(u,v)\in E(G)$ and also $D$ is a total dominating set if for every vertex $u \in V(G)$, there is a vertex $v\in D$  such that $(u,v)\in E(G)$. A Roman dominating function on a graph $G=(V, E)$ is a function $ f : V \rightarrow \{0, 1, 2\}$ satisfying
the property that every vertex $u$ for which $f(u) = 0$ is adjacent to at least one vertex $v$ for which $f(v) = 2$. The weight of a Roman dominating function (RDF) is $f(V) = \sum_{u\in V}{f(u)}$.
The minimum weight of a Roman dominating function on a graph $G$ is called the Roman
domination number of $G$. A \textit{parameterized problem} is a tuple $(\mathcal{L},k)$ where $\mathcal{L}$ is a decision problem over some finite alphabet $\Sigma$ and $k:\Sigma^* \rightarrow \mathbb{N}$ is a parameter for each instance of $\mathcal{L}$.
An algorithm is \textit{fixed-parameter tractable} or \textit{fpt} if its running time is at most $f(k)\cdot n^{O(1)}$ for some arbitrary function $f$, where $n$ is the input size and $k$ is the parameter assigned to the input.
Given two parameterized problems $(\mathcal{L}_1,k_1)$ and $(\mathcal{L}_2,k_2)$ over some finite alphabet $\Sigma$, an \textit{fpt-reduction} from $(\mathcal{L}_1,k_1)$ to $(\mathcal{L}_2,k_2)$ is a function $g:\Sigma^* \rightarrow \Sigma^*$ such that $I\in \mathcal{L}_1$ iff $g(I)\in \mathcal{L}_2$ and $k_2 (g(I)) \leq f(k_1(I))$ for every $I\in \Sigma^*$, where $f$ is an arbitrary function.

\section{$r$-Hop Roman Domination (Proof of Theorem~\ref{thm:combined_r_roman_hop})}\label{Roman_Hop_Domination}


First, we show that the \textsc{$r$-Hop Roman Domination} problem is in \textsc{W[2]}. To show that we reduce our problem to the \textsc{roman domination} problem, which is \textsc{W[2]}-complete~\cite{fernau2008roman}. 

\medskip
\noindent
\fbox{%
  \begin{minipage}{\dimexpr\linewidth-2\fboxsep-2\fboxrule}
  \textsc{ Roman Domination (RD)} problem
  \begin{itemize}[leftmargin=1.5em]
    \item \textbf{Input:} An undirected graph $G = (V, E)$ and an integer $k \in \mathbb{N}$.
    \item \textbf{Question:} Does there exist a \emph{roman dominating function} on $G$ of weight at most $k$.
  \end{itemize}
  \end{minipage}%
}

\begin{lemma}\label{roman hop w[2]}
   For each $r\geq 2$, the \textsc{$r$-Hop Roman Domination} problem is in \textsc{W[2]}.
\end{lemma}

\begin{proof}
    Let $(G,k)$ be an instance of the \textsc{$r$-Hop Roman Domination} problem. Let $V(G) = \{v_1, v_2, \dots, v_n\}$. We construct a new graph $G'$ as follows:
    
    \begin{enumerate}[]
        \item For each vertex $v_i \in V(G)$, introduce a corresponding vertex $u_i \in V(G')$.
        \item For every pair of vertices $v_j, v_k \in V(G)$ such that $d_G(v_j,v_k)=r$, 
we add the edge $(u_j,u_k)$ to $E(G')$.
    \end{enumerate}

    We now show that $G$ has an  $r$HRDF of weight at most $k$ if and only if $G'$ has an RDF of weight at most $k$.
    
    Let $f_1=(P,Q,R)$ be an $r$HRDF of $G$. Then $f_2=(P',Q',R')$ is a RDF of $G'$, where, if $(v_i,v_j,v_k)\in (P,Q,R)$, then corresponding vertex of  $(v_i,v_j,v_k)$ which is  $(u_i,u_j,u_k)$, belongs to $(P',Q',R')$. If not, there exists a vertex $u_l \in P'$ for which there is no $u_m\in Q'$ such that $d(u_m,u_l)=r$. This contradicts that $f_1$ is an $r$HRDF. 
    Conversely,$f_2=(P',Q',R')$ is a RDF of $G'$, where, if $(u_i,u_j,u_k)\in (P',Q',R')$, then corresponding vertex of  $(u_i,u_j,u_k)$ which is  $(v_i,v_j,v_k)$, belongs to $(P,Q,R)$. If not, there exist a vertex $v_l \in P$ for which there is no $v_m\in Q$ such that $d(v_m,v_l)=1$. This contradicts that $f_2$ is an RDF.
   \end{proof} 

To prove \textsc{$r$-Hop roman domination} problem \textsc{W[2]}-hard we reduce from \textsc{Domination} problem in general graph, which is \textsc{W[2]}-hard to our problem.
\begin{figure}[h]
    \centering
   \includegraphics[width=0.8\textwidth]{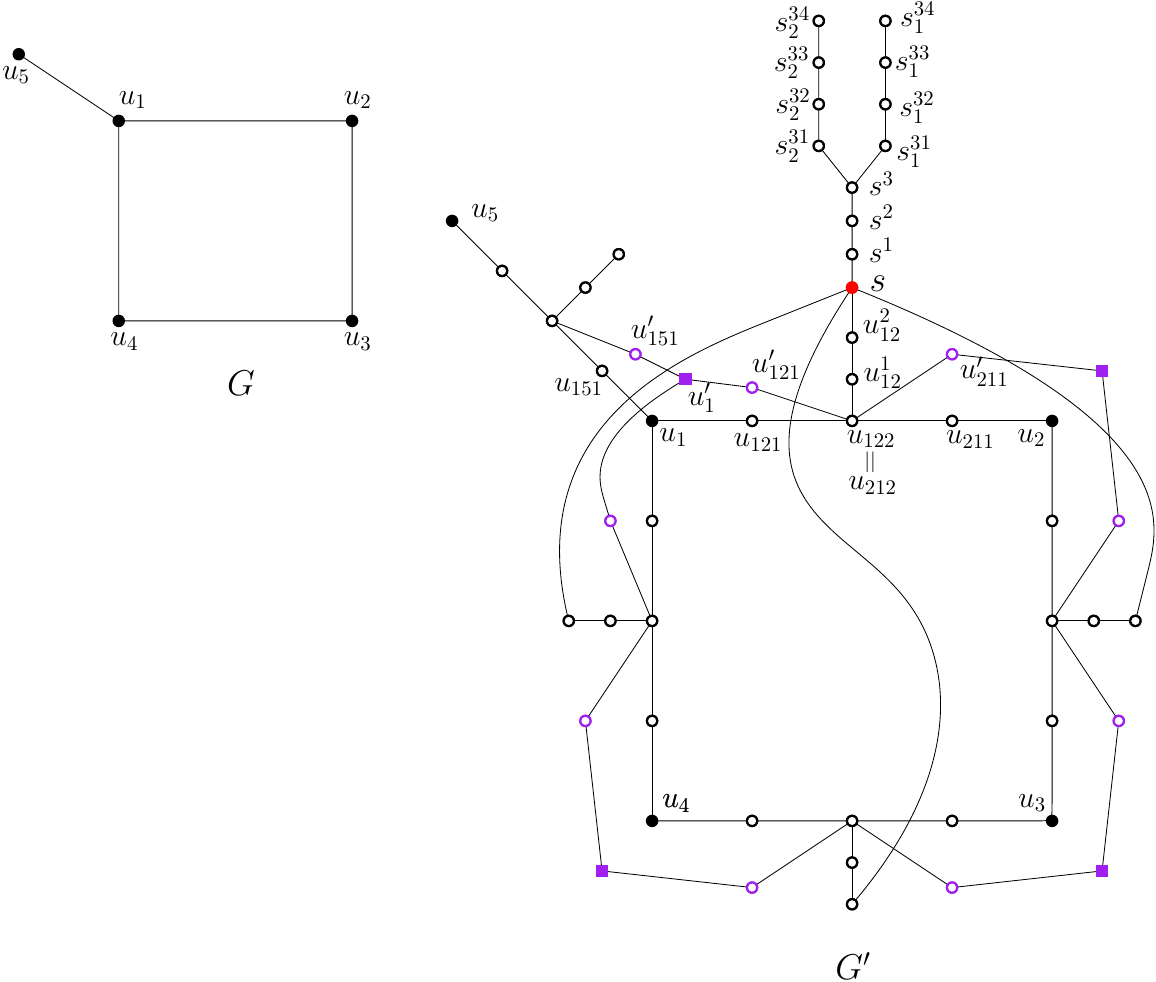}
    \caption{This is a pictorial description of the gadget construction of $G'$ from $G$, for $4$-Hop Roman Domination problem.}
    \label{roman hop hardness}
\end{figure}

\begin{lemma}\label{roman hop even W[2] complete}
  For any even number  $r\geq2$, the \textsc{$r$-Hop Roman Domination} problem is \textsc{W[2]}-complete.
  
\end{lemma}
\begin{proof}
    It is known that the \textsc{Dominating Set} problem is \textsc{W[2]}-complete.  We reduce the \textsc{Domination} problem to to the \textsc{$r$-Hop Roman Domination} problem to show that \textsc{$r$-Hop Roman Domination} is \textsc{W[2]}-hard for any even $r$.
    Let $(G,k)$ be an instance of the \textsc{Dominating Set} problem. 
    Let $V(G) = \{v_1,v_2,\dots,v_n\}$, an edge set $E(G)$. We construct a new graph $G'$ in the following way:
(i) We take a copy of \( G \), say $G'$. We denote the vertex set of \( G' \) as \( V(G') = \{u_1, u_2, \dots, u_n\} \), where each vertex \( u_i \in V(G') \) corresponds to \( v_i \in V(G) \). (ii) For each edge $(u_i,u_j)\in E(G')$ with $i<j$, we subdivide it $(r-1)$-times by introducing $(r-1)$ new intermediate vertices $u_{ij1},u_{ij2},\dots,u_{ij(r-1)}=u_{ji(r-1)},u_{ji(\frac{r}{2}-1)},u_{ij(\frac{r}{2}-2)},\dots,u_{ji2},u_{ji1}$ and forming the path: $u_i-u_{ij1}-u_{ij2}-\dots-u_{ij(r-1)}=u_{ji(r-1)}-u_{ji(\frac{r}{2}-1)}-u_{ij(\frac{r}{2}-2)}-\dots-u_{ji2}-u_{ji1}-u_j$.
(iii) We add a universal vertex $s$ and draw the edges $(s,u_{ijr/2})$.
(iv) For each edge $(s,u_{ijr/2})$, we subdivide it $\frac{r}{2}$-times by introducing $r/2$ new intermediate vertices $u^1_{ij},u^2_{ij},\dots,u^{r/2}_{ij}$, where $i<j$ and forming the path: $u_{ijd/2}-u^1_{ij}-u^2_{ij}-\dots-u^{r/2}_{ij}-s$. (v) At $s$ we draw a $(r-1)$ length path including $s$ by introducing $r-1$ new vertices $s^1,a^2,\dots,s^{r-1}$ and forming the path $s-s^1-s^2-\dots-s^{r-1}$. (vi) At $s^{r-1}$ we draw two $r$-length path including $s^{r-1}$  by introducing $2r$ vertices $s^{(r-1)1}_1,s^{(r-1)2}_1,\dots,s^{(r-1)r}_1$ and $s^{(r-1)1}_2,s^{(r-1)2}_2,\dots,s^{(r-1)r}_2$ and forming the path $s^{r-1}-s^{(r-1)1}_1-s^{(r-1)2}_1-\dots-s^{(r-1)r}_1$ and $s^{r-1}-s^{(r-1)1}_2-s^{(r-1)2}_2-\dots-s^{(r-1)r}_2$. (vii) For each vertex $u_i \in V(G')$ we include a vertex $u'_i$. (viii) Join each $\frac{r}{2}$-th neighbor of $u_i$ to $u'_i$ by  an edge, then we consider all these new edges and for this such edge $(u'_i,u_{ij\frac{r}{2}})$, we subdivide it $(\frac{r}{2}-1)$-th times by introducing the vertices $u'_{ij1},u'_{ij2},\dots, u'_{ij(\frac{r}{2}-1)}$, where there is a path $u_i-u_{ij1}-u_{ij2}-\dots-u_{ij(\frac{r}{2}-1)}-u_{ij(\frac{r}{2})}$ between $u_i$ and $u_{ij\frac{r}{2}}$.

\begin{claim}
    $G$ has a dominating set of size at most $k$ if and only if $G'$ has an $r$HRDF of weight at most $2k+2r$, where $r$ is an even number and  $r\geq2$.
\end{claim}
\begin{claimproof}
   \textsf{(if part)} Suppose $G$ has a dominating set $\{v_1,v_2,\dots,v_k\}$. We consider the set $H=\{u_1,u_2,\dots,u_k,s,s^1,s^2.\dots,s^{r-1}\}$. We set $f=\{V(G')-H,\phi,H\}$. We conclude that any
vertex $v \in V(G')$ with $f(v) = 0$ is hop dominated by a vertex $u\in V(G')$ with $f(u) = 2$. Hence $G'$ has an $r$HRDF of weight at most $2k+2r$. 

\textsf{(Only if part)} Let $f=(V_0^f,V_1^f,V_2^f)$ be an $r$HRDF of weight $\leq 2k+2r$ of $G'$. Now if $f (s)+f(s^1)+f(s^2)+\dots+f (s^{r-1})+f(s^{(r-1)1}_1)+f(s^{(r-1)2}_1)+\dots+f(s^{(r-1)r}_1)+f(s^{(r-1)1}_2)+f(s^{(r-1)2}_2)+\dots+f(s^{(r-1)r}_2) < 2r$, there is a vertex in $\{s^{(r-1)1}_1,s^{(r-1)2}_1,\dots,s^{(r-1)r}_1,s^{(r-1)1}_2,$ $s^{(r-1)2}_2,\dots,s^{(r-1)r}_2\}$ such that it is not roman hop dominated by $f$.  Now we denote the set $\{s,s^{(r-1)1}_1,s^{(r-1)2}_1,\dots,s^{(r-1)r}_1,s^{(r-1)1}_2,s^{(r-1)2}_2,\dots,s^{(r-1)r}_2\}$ as $S$. Hence $\sum_{v\in V(G')/\{S\}}^{}f(v)\leq 2k$. We can modify $f$ to contain only vertices from $\{u_1, \dots, u_n\}$ in the following way:
\begin{itemize}
    \item If $f(u^{\frac{r}{2}}_{ij}) = 1$ and  $f(u_i') = 1$ for some $k, m$, then $f$ will be updated by $f=(V^f_0\cup \{ u_i',u^{\frac{r}{2}}_{ij}\},V^f_1\setminus\{u_i',u^{\frac{r}{2}}_{ij}\},V^f_2\cup \{u_i\})$.
    
    \item if $f(u^{\frac{r}{2}}_{ij})=2$ then $f$ will be updated by $f=(V^f_0 \cup \{u^{\frac{r}{2}}_{ij},\},V^f_1,V^f_2\cup \{u_i$ or $u_j\})$. 
    \item If $f(u_i)=1$ and $f(u_i')=1$, then $f$ will be updated by $f=(V^f_0 \cup \{u_i'\},V^f_1\setminus\{u_i\},V^f_2 \cup \{u_i \})$.
    \item If $f(u_i')=2$ then $f$ will be updated by $f=(V^f_0 \cup \{u_i'\},V^f_1,V^f_2 \cup \{u_i \})$.  
\end{itemize}

 Let  $T=\{v_i\in G : f(u_i)=2\}$. We show that $T$ is a dominating set of size $\leq k$ of $G$

    let $T = \{v_i \in G : f(u_i)=2\}$ is not a dominating set of $G$. That is, there exist a vertex $v_k \in G$ which is not dominated, hence the corresponding vertex $u_k (\in G') =0$ or $1$.
    
    \textbf{\emph{Case 1:}} Let $u_k \in V(G') =0$. Then there exit a vertex $x\in V(G')$ such that $f(x)=2$ and $d(u_k,x)=r$ since $G'$ is roman hop dominated by $f$. Now, this $f(x)=2$ can be changed to $f(x)=0$. Then we can choose a vertex $u_l \in V(G')$ in such a way that $d(u_k,u_l)=r$, and also $N_r(u_l)=N_r(x)$ and set $f(u_l)=2$. Now the corresponding vertex of $u_l$ which is $v_l$ is adjacent to $v_k$ or $v_l=v_k$, which is a contradiction. 
    
    \textbf{\emph{Case 2:}}  Let $f(u_k) =1$, where $u_k\in V(G')$.
    
    {\emph{Case i(a):}} If there exist a vertex $u_k'$  such that $f(u_i')=1$ and $d(u_k',u_k)=r$. Now $f(u_k) =1$ can be changed to $f(u_k) =2$ and $f(u_k')=1$ can be also changed to $f(u_k')=0$. Hence $u_k \in T$, which is a contradiction.
    
    {\emph{Case i(b):}}  If there exist a vertex $u^{\frac{r}{2}}_{kj}$, for some $j$ such that $f(u^{\frac{r}{2}}_{kj})=1$ and $d(u^{\frac{r}{2}}_{kj},u_k)=r$. Now $f(u_k) =1$ can be changed to $f(u_k) =2$ and $f(u^{\frac{r}{2}}_{kj})=1$ can be also changed to $f(u^{\frac{r}{2}}_{kj})=0$. Hence $u_k \in T$, which is a contradiction.

     {\emph{Case ii:}} If $u_k'$ such that $d(u_k',u_k)=r$ $f(u_k')=0$. Then there exit a vertex $x\in V(G')$ such that $f(x)=2$ and $d(u_k',x)=r$, if $x=u_k$ then we are done. If $x=u_l$, where $l \neq k$, then  corresponding vertex of $u_l$ in $G$ i.e $v_l$ is surely adjacent with $v_k$. If $x \notin \{u_1,u_2\dots,u_n\}$ then $f(x)=2$ can be changed to $f(x)=0$ and choose a $u_l$ such that $d(x,u_l)=r$ and then set $f(u_l)=2$. Clearly $u_l=u_k$ or corresponding vertices of $u_k$ and $u_l$ in $G$ are adjacent. This contradicts the fact $v_k \in V(G)$ which is not dominated.  
\end{claimproof}

Therefore the \textsc{$r$-Hop roman domination} problem, where $r$ is even is \textsc{W[2]}-hard. By lemma~\ref{roman hop w[2]}, we have that the problem is in \textsc{W[2]}. Hence the problem is \textsc{W[2]}-complete.
\end{proof}

\begin{figure}[h]
    \centering
   \includegraphics[width=0.6\textwidth]{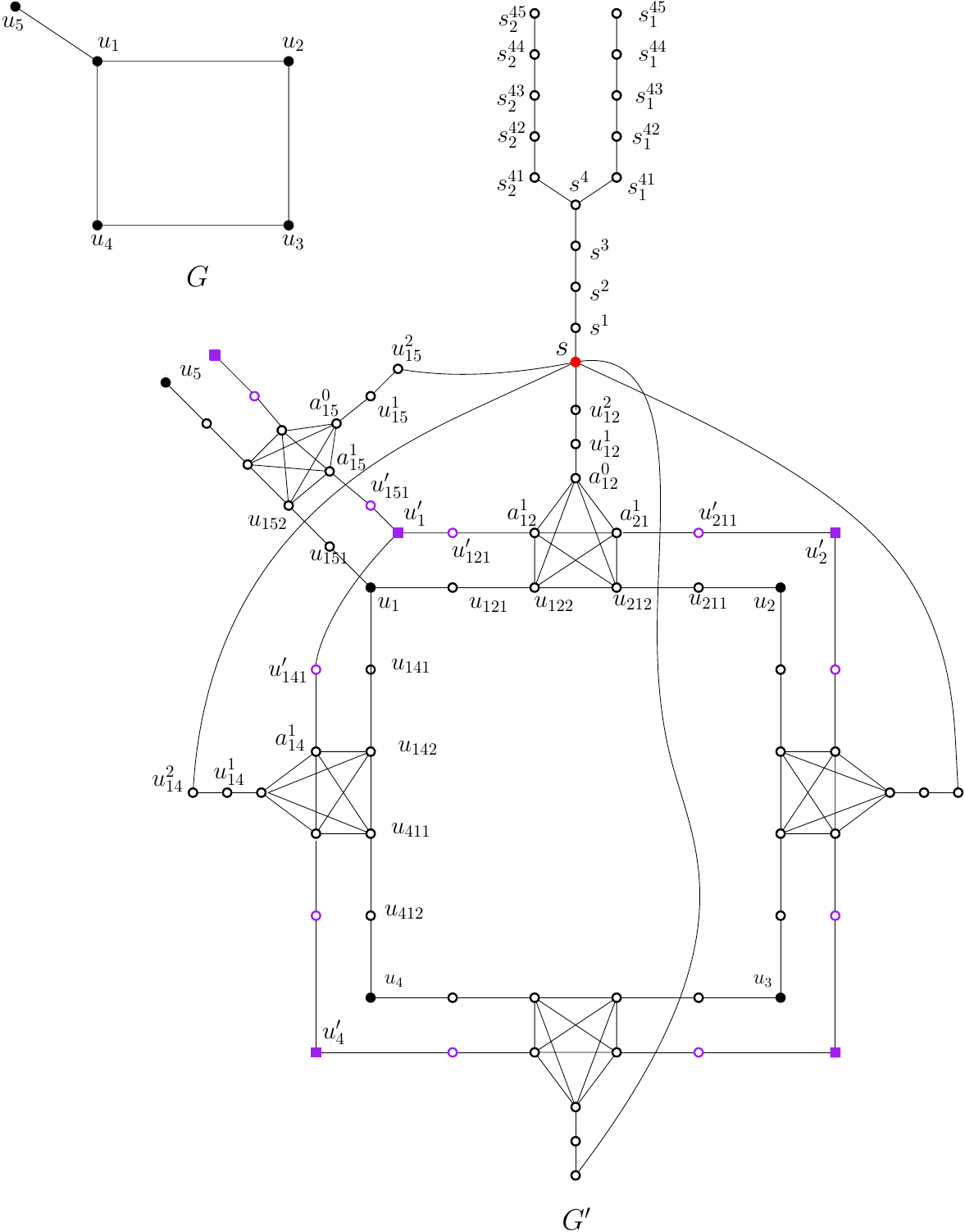}
    \caption{This is a pictorial description of the gadget construction of $G'$ from $G$, for \textsc{$5$-Hop Roman Domination} problem.}
    \label{roman hop hardness}
\end{figure}

\begin{lemma}
For any odd number  $r\geq 3$, the    \label{roman hop odd W[2] complete}
    \textsc{$r$-Hop Roman  Domination} problem is \textsc{W[2]}-complete.
    
\end{lemma}
\begin{proof}

It is known that the \textsc{Dominating Set} problem is \textsc{W[2]}-complete.  We reduce the \textsc{Domination} problem to the \textsc{$r$-Hop Domination} problem to show that the \textsc{$r$-Hop Domination} problem is \textsc{W[2]}-hard for any odd $r\geq 3$.
 
    Let $(G,k)$ be an instance of the \textsc{ Dominating Set} problem. Let $V(G) = \{v_1,v_2,\dots,v_n\}$, an edge set $E(G)$ and a dominating set $D(G)$ of size at least $k$, we construct a new graph $G'$ in the following way:
(i) We start by taking a copy of \( G \), denoting the vertex set of \( G' \) as \( V(G') = \{u_1, u_2, \dots, u_n\} \), where each vertex \( u_i \in V(G') \) corresponds to \( v_i \in V(G) \). 
(ii) For each edge $(u_i,u_j)\in E(G')$ with $i<j$, we replace it with a $(d+3)$-length path by introducing $d+1$ new intermediate vertices $u_{ij1},u_{ij2},\dots,u_{ij(\frac{r-1}{2})},u_{ji(\frac{r+1}{2})},u_{ji(\frac{r+1}{2}+1)},\dots,u_{ji2},u_{ji1}$ and forming the path $L_{ij}$: $u_i-u_{ij1}-u_{ij2}-\dots-u_{ij(\frac{r-1}{2})}-u_{ji(\frac{r+1}{2})}-u_{ji(\frac{r+1}{2}+1)}-\dots-u_{ji2}-u_{ji1}-u_j$. (iv) For every $L_{ij}$ we select $u_{ij(\frac{r-1}{2})}$ and $u_{ij(\frac{r+1}{2})}$ and also  introduce three new vertices $a^1_{ij}, a^1_{ji},a^0_{ij}=a^0_{ji}$ and create a $K_5$ (namely $A_{ij}$) with these new vertices and $u_{ij(\frac{r-1}{2})},u_{ji(\frac{r+1}{2})}$ vertices of $L_{ij}$.
(v) We add a universal vertex $s$ and draw the edges $(s,a^5_{ij})$.
(v) For each edge $(s,a^0_{ij})$, we replace it with a $(r-1)/2+2$-length path by introducing $(r-1)/2$ new intermediate vertices $u^1_{ij},u^2_{ij},\dots,u^{(r-1)/2}_{ij}$, where $i<j$ and forming the path: $a^0_{ij}-u^1_{ij}-u^2_{ij}-\dots-u^{(r-1)/2}_{ij}-s$.
(vi) At $s$ we add a $(r-1)$ length path including $s$ by introducing $r-1$ new vertices $s^1,s^2,\dots,s^{r-1}$ and forming the path $s-s^1-s^2-\dots-s^{r-1}$ . (vii) At $s^{r-1}$ we add two $r$-length path including  $s^{r-1}$  by introducing $2r$ vertices $s^{(r-1)1}_1,s^{(r-1)2}_1,\dots,s^{(r-1)d}_1$ and $s^{(r-1)1}_2,s^{(r-1)2}_2,\dots,s^{(r-1)r}_2$ and forming the path $s^{r-1}-s^{(r-1)1}_1-s^{(r-1)2}_1-\dots-s^{(r-1)r}_1$ and $s^{r-1}-s^{(r-1)1}_2-s^{(r-1)2}_2-\dots-s^{(r-1)r}_2$. (viii) For each vertex $u_i \in V(G')$ we include a vertex $u'_i$. (ix) Join each $a^1_{ij}$ to $u'_i$ by  an edge, then we consider all these new edges and for this such edge $(u'_i,a^1_{ij})$, we subdivide it $(\frac{r-1}{2}-1)$-th times by introducing the vertices $u'_{ij1},u'_{ij2},\dots, u'_{ij(\frac{r-1}{2})}$, and forming the path $u_i'-u'_{ij1}-u'_{ij2}-\dots- u'_{ij(\frac{r-1}{2})}-a^1_{ij}$  when there is a path $u_i-u_{ij1}-u_{ij2}-\dots-u_{ij(\frac{r-1}{2}-1)}-u_{ij(\frac{r-1}{2})}$ between $u_i$ and $u_{ij(\frac{r-1}{2})}$.




\begin{claim}
    \label{lemma roman hop}
    $G$ has a dominating set of size at most $k$ if and only if $G'$ has an $r$HRDF of weight at most $2k+2r$, where $r$ is an odd number with $\geq3$.
\end{claim}
\begin{claimproof}
    \textsf{(if part)} Suppose $G$ has a  dominating set $\{v_1,v_2,\dots,v_k\}$. We consider the set $H=\{u_1,u_2,\dots,u_k,s,$ $s^1,s^2.\dots,s^{r-1}\}$. We set $f=\{V(G')-H,\phi,H\}$. We conclude that any
vertex $v \in V(G')$ with $f(v) = 0$ is hop dominated by a vertex $u\in V(G')$ with $f(u) = 2$. Hence $G'$ has an $r$HRDF of weight at most $2k+2r$. 

\textsf{(Only if part)} Let $f=(V_0^f,V_1^f,V_2^f)$ is an $r$HRDF of weight at most $\leq 2k+2r$ of $G'$. Now if $f (s)+f(s^1)+f(s^2)+\dots+f (s^{r-1})+f(s^{(r-1)1}_1)+f(s^{(r-1)2}_1)+\dots+f(s^{(r-1)r}_1)+f(s^{(r-1)1}_2)+f(s^{(r-1)2}_2)+\dots+f(s^{(r-1)r}_2) < 2r$, there is a vertex in $\{s^{(r-1)1}_1,s^{(r-1)2}_1,\dots,s^{(r-1)r}_1,s^{(r-1)1}_2,$ $s^{(r-1)2}_2,\dots,s^{(r-1)r}_2\}$ such that it is not roman hop dominated by $f$.  Now we denote the set $\{s,s^{(r-1)1}_1,s^{(r-1)2}_1,\dots,s^{(r-1)r}_1,s^{(r-1)1}_2,s^{(r-1)2}_2,\dots,s^{(r-1)r}_2\}$ as $S$. Hence $\sum_{v\in V(G')/\{S\}}^{}f(v)\leq 2k$. We can modify $f$ to contain only vertices from $\{u_1, \dots, u_n\}$ in the following way:
\begin{itemize}
    \item If $f(u_i') = 1$ and  $f(u_i') = 1$  then $f$ will be updated by $f=(V^f_0\cup \{ u_i',u^{\frac{r}{2}}_{ij}\},V^f_1\setminus\{u_i',u^{\frac{r}{2}}_{ij}\},V^f_2\cup \{u_i\})$.
    
    \item if $f(u^{(\frac{r-1}{2})}_{ij})=2$ then $f$ will be updated by $f=(V^f_0 \cup \{u^{(\frac{r-1}{2})}_{ij},\},V^f_1,V^f_2\cup \{u_i$ or $u_j\})$. 
    \item If $f(u_i)=1$ and $f(u_i')=1$ then $f$ will be updated by $f=(V^f_0 \cup \{u_i'\},V^f_1\setminus\{u_i\},V^f_2 \cup \{u_i \})$.
    \item If $f(u_i')=2$  then $f$ will be updated by $f=(V^f_0 \cup \{u_i'\},V^f_1,V^f_2 \cup \{u_i \})$.  
\end{itemize}

     Let $T = \{v_i \in G : f(u_i)=2\}$. We show that $T$ is a dominating set of size $\leq k$ of $G$.

    let $T = \{v_i \in G : f(u_i)=2\}$ is not a dominating set of $G$. That is, there exists a vertex $v_k \in G$ which is not dominated, hence the corresponding vertex $u_k (\in G') =0$ or $1$.
    
    \textbf{\emph{Case 1:}} Let $u_k (\in G') =0$. Then there exit a vertex $x\in G'$ such that $f(x)=2$ and $d(u_k,x)=r$ since $G'$ is roman hop dominated by $f$. Now, this $f(x)=2$ can be changed to $f(x)=0$. Then we can choose a vertex $u_l \in V(G')$ in such a way that $d(u_k,u_l)=r$, and also $N_r(u_l)=N_r(x)$ and set $f(u_l)=2$. Now the corresponding vertex of $u_l$, which is $v_l$, is adjacent to $v_k$ or $v_l=v_k$, which is a contradiction. 
    
    \textbf{\emph{Case 2:}}  Let $f(u_k) =1$, where $u_k\in V(G')$.
    
    {\emph{Case i(a):}} If there exist a vertex $u_k'$  such that $f(u_k')=1$ and $d(u_k',u_k)=r$. Now $f(u_k) =1$ can be changed to $f(u_k) =2$ and $f(u_k')=1$ can be also changed to $f(u'_k)=0$. Hence $u_k \in T$, which is a contradiction.
    
    {\emph{Case i(b):}}  If there exist a vertex $u^{(\frac{r-1}{2})}_{kj}$, for some $j$ such that $f(u^{(\frac{r-1}{2})}_{kj})=1$ and $d(u^{(\frac{r-1}{2})}_{kj},u_k)=r$. Now $f(u_k) =1$ can be changed to $f(u_k) =2$ and $f(u^{(\frac{r-1}{2})}_{kj})=1$ can be also changed to $f(u^{(\frac{r-1}{2})}_{kj})=0$. Hence $u_k \in T$, which is a contradiction.

     {\emph{Case ii: }}If $u_k'$ such that $d(u_k',u_k)=r$, $f(u_k')=0$. Then there exits a vertex $x\in V(G')$ such that $f(x)=2$ and $d(u_k',x)=r$, if $x=u_k$ then we are done. If $x=u_l$, where $l \neq k$, then  the corresponding vertex of $u_l$ in $G$ i.e $v_l$ is surely incident to $v_k$. If $x \notin \{u_1,u_2\dots,u_n\}$ then $f(x)=2$ can be changed to $f(x)=0$ and choose a $u_l$ such that $d(x,u_l)=r$ and then set $f(u_l)=2$. Clearly, $u_l=u_k$ or the corresponding vertices of $u_k$ and $u_l$ in $G$ are adjacent. This contradicts the fact that $v_k \in V(G)$, which is not dominated. 
\end{claimproof}
Therefore the \textsc{$r$-Hop roman dominating} problem, where $r$ is odd and $r\geq3$ is \textsc{W[2]}-hard. By Lemma~\ref{roman hop w[2]}, we have that the problem is in \textsc{W[2]}. Hence, the problem is \textsc{W[2]}-complete.
\end{proof}

 From Lemma~\ref{roman hop even W[2] complete} and Lemma~\ref{roman hop odd W[2] complete} we get the following theorem:
 \begin{restatable}{lemma}{rromanhopdomination}\label{thm:$r$-roman hop domination} For each $r\geq 2$, the \textsc{$r$-Hop Roman Domination} problem is \textsc{W[2]}-complete.
\end{restatable}
 
Unless ETH fails, the \textsc{Dominating Set} problem admits an algorithm running in time $2^{o(n+m)}$, where $n$ and $m$ are
the cardinalities of the vertex and edge sets of the input graph, respectively~\cite{cygan2015parameterized}. From Lemma~\ref{roman hop even W[2] complete} and Lemma~\ref{roman hop odd W[2] complete} we get the following result:

\begin{corollary}
   Unless ETH fails, the \textsc{$r$-Hop Roman Domination} problem does not have $2^{o(n+m)}$-time algorithm, where $n $ is the number of vertices and $m$ is the number of edges of the graph.
\end{corollary}
The classical \textsc{Dominating Set} problem cannot
be approximated to within a factor of $(1 - \epsilon)\log n$ in polynomial time for any constant $\epsilon > 0$ unless $NP \subseteq DTIME(n^{O(\log \log n})$~\cite{chlebik2008approximation}.  From Lemma~\ref{roman hop even W[2] complete} and Lemma~\ref{roman hop odd W[2] complete} we get the following result:

\begin{corollary}
The   \textsc{$r$-Hop Roman Domination} problem can not be approximated to within a factor of $(1- \epsilon)\log$ $n$, where $n$ is the number of vertices of the graph.
\end{corollary}
From Lemma~\ref{roman hop w[2]} we get the following result:

\begin{corollary}\label{exp algo hop}
If there exists an $O(\alpha^n)$-time algorithm for computing a minimum 
{dominating set}, then there also exists an $O(\alpha^n)$-time 
algorithm for computing a minimum {$r$-hop Roman dominating set}, where $n$ is the number of vertices of the graph.
\end{corollary}

\noindent\textbf{Remark.}
Currently, the fastest known exact algorithm for 
\textsc{Dominating Set} is due to Iwata~\cite{iwata2011faster}, 
which runs in $O(1.4864^n)$ time using polynomial space. 
Therefore, by Corollary~\ref{exp algo hop}, 
there exists an $O(1.4864^n)$-time algorithm for computing a minimum 
{$r$-hop Roman dominating set} on an $n$-vertex graph.

\section{$r$-Step Domination (Proof of Theorem~\ref{thm:$r$-step domination restricted})}\label{sec:rStepDomination}


    
    

First we prove that the \textsc{$r$-Step Domination} problem is in \textsc{W[2]}. To show that we reduce our problem to the classical \textsc{Total Dominating Set} problem, which is \textsc{W[2]}-complete~\cite{henning2013total}.

\medskip
\noindent
\fbox{%
  \begin{minipage}{\dimexpr\linewidth-2\fboxsep-2\fboxrule}
  \textsc{ Total Dominating Set} problem
  \begin{itemize}[leftmargin=1.5em]
    \item \textbf{Input:} An undirected graph $G = (V, E)$ and an integer $k \in \mathbb{N}$.
    \item \textbf{Question:} Does there exist a \emph{total dominating set} $S \subseteq V$ of size at most $k$; that is, a set of at most $k$ vertices such that every vertex in $V$ has at least one neighbor in $S$?
  \end{itemize}
  \end{minipage}%
}
\medskip

\begin{lemma}\label{r step w[2]}
For each $r\geq 2$, the \textsc{$r$-Step Domination} problem is in \textsc{W[2]}.
\end{lemma}

\begin{proof}
Let $(G,k)$ be an instance of the \textsc{$r$-Step Domination} problem. Let $V(G) = \{v_1, v_2, \dots, v_n\}$, we construct a new graph $G'$ as follows:

    \begin{enumerate}[]
        \item For each vertex $v_i \in V(G)$, introduce a corresponding vertex $u_i \in V(G')$.
        \item For every pair of vertices $v_j, v_k \in V(G)$ such that $d_G(v_j,v_k)=r$, 
we add the edge $(u_j,u_k)$ to $E(G')$.
    \end{enumerate}
    
    
    We now show that $G$ has a $r$-step dominating set of size $k$ if and only if $G'$ has a total dominating set of size $k$.
     Let $G$ has a $r$-step dominating set $S_1= \{v_1,v_2,\dots,v_k\}$, then $S_2=\{u_1,u_2,\dots,u_k\}$ is a total dominating set. If not, there exists a vertex $u_l \in V(G')$, which is not dominated by $S_2$. This contradicts $S_1$ is an $r$-step dominating set of $G$.
    Conversely, let  $G'$ has a total dominating set $S_2= \{u_1,u_2,\dots,u_k\}$, then $S_1= \{v_1,v_2,\dots,v_k\}$ is a $r$-step dominating set. If not, there exists a vertex $v_l \in V(G)$ which is not $r$-step dominated by $S_1$. This contradicts $S_2$ is a total dominating set of $G$.
\end{proof}

To prove that \textsc{$r$-Step  Domination} Problem is \textsc{W[2]}-hard, even for bipartite graphs and also chordal graphs, we reduce from the \textsc{Domination} problem in general graphs, which is \textsc{W[2]}-hard.

\medskip
\noindent
\fbox{%
  \begin{minipage}{\dimexpr\linewidth-2\fboxsep-2\fboxrule}
  \textsc{ Dominating Set} problem
  \begin{itemize}[leftmargin=1.5em]
    \item \textbf{Input:} An undirected graph $G = (V, E)$ and an integer $k \in \mathbb{N}$.
    \item \textbf{Question:} Does there exist a \emph{dominating set} $S \subseteq V$ of size at most $k$?
  \end{itemize}
  \end{minipage}%
}
\medskip

\begin{restatable}{lemma}{rstepdominationbipartite}\label{thm:$r$-step domination bipartite} For $r\geq 2$, \textsc{$r$-Step Domination} Problem is \textsc{W[2]}-hard, even on bipartite graphs.
\end{restatable}

\begin{figure}[h]
    \centering
   \includegraphics[width=\textwidth]{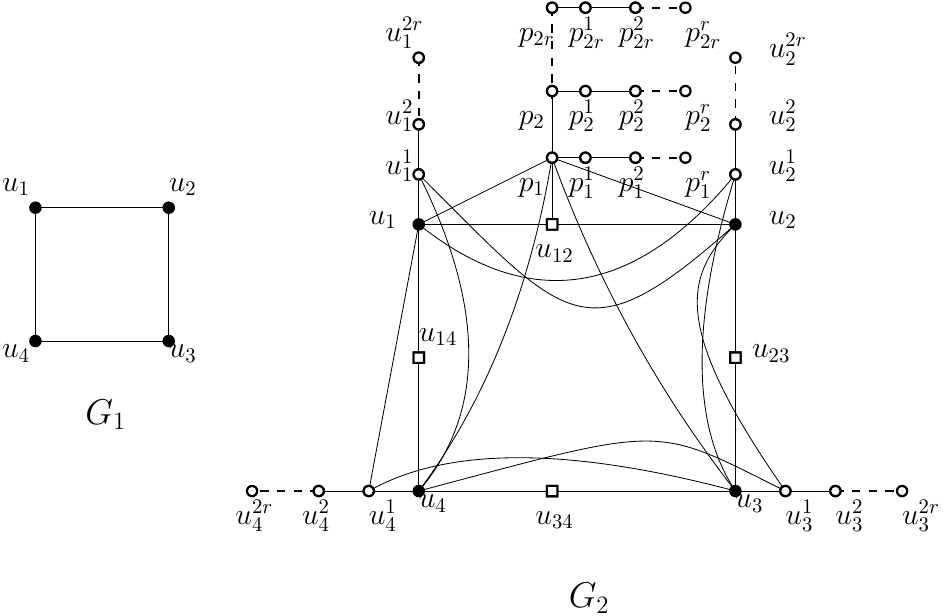}
    \caption{This is a pictorial description of the gadget construction of $G_2$ from $G_1$.}
    \label{r step hardness}
\end{figure}
\begin{proof} Given a graph $G_1$ with a vertex set $V(G_1) = \{v_1,v_2,\dots,v_n\}$, an edge set $E(G_1)$. We construct a new graph $G_2$ in the following way:
(i) We start by taking a copy of \( G_1 \) and denote the vertex set of \( G_2 \) as \( V(G_2) = \{u_1, u_2, \dots, u_n\} \), where each vertex \( u_i \in V(G_2) \) corresponds to \( v_i \in V(G_1) \). (ii) For each edge $(u_i,u_j)\in E(G_2)$ with $i<j$, we subdivide the edge by introducing a new vertex  $a_{ij}$. (iii) For each vertex $u_i$ we add $r$ new vertices $u_i^1,u_i^2,\dots,u_i^r$ and draw a path $u_i-u_i^1-u_i^2-\dots-u_i^r$. (iv) If there is an edge between $(v_i,v_j)\in E(G_1)$ we draw the edges $(u_i^1,u_j)$ and $(u_i,u_j^1)$. (v) We add another vertex $p_1$ and draw the edges $(u_i,p_1)$ for all $i=1,2,\dots,n$. (vi) We introduce $2r-1$ new vertices $p_2,p_2,\dots,p_{2r}$ and draw the path $p_2-p_3-\dots-p_{2r}$. (vii) We add a $r$-length path for each $p_i$  by introducing $r$ new vertices $p^1_i,p^2_i,\dots,p^r_i$ and forming the path $p_i-p^1_i-p^2_i-\dots-p^r_i$ and then join $p_i^1$ with $p_i$ by an edge. 

 \begin{claim}\label{claim r step bipartite}
       $G_1$ has a dominating set of size at most $k$, if and only if $G_2$ has an $r$-step dominating Set of size at most $k + 2r$, where $r\geq 2$.
    \end{claim}
     
    \begin{claimproof}
        \textsf(If part) Suppose $G_1$ has a dominating set $T$ with $|T| \leq k$. Let $S_1 = \{p_1, p_2,\dots, p_{2r}\}$. This set $r$-step dominates all vertices in $V(G_2) \setminus \{u_i^r \mid i = 1, \dots, n\}$. The remaining vertices $\{u_i^r \mid i = 1, \dots, n\}$ are $r$-step dominated by $T' = \{u_i \mid v_i \in T\}$. Thus, $G_2$ has a $r$-step dominating Set of size $\leq k + 2r$.

         \textsf{(Only-if part)} Suppose $G_2$ has a $r$-step dominating Set $D$ with $|D| \leq k + 2r$. In order to $r$-step dominate the vertex $p_i^r \in G_2$, $p_i$ must be in $D$. The set $\{p_1,p_2,\dots,p_{2r}\}$ hop dominate all the vertices in $V(G_2)$ except $\{u_1^r,u_2^r,\dots,u_n^r\}$. Thus the vertices $\{u_i^r \mid i = 1, \dots, n\}$ are $r$-step dominated by a set $T' \subseteq D$ with $|T'| \leq k$.
             We can modify $T'$ to contain only vertices from $\{u_1, \dots, u_n\}$. 
             Let $T = \{v_i \mid u_i \in T'\}$. We show that $T$ is a dominating set for $G_1$.

        \medskip

         Suppose there exists $v_k \in V(G_1)$ not dominated by $T$. Then $u_k^r \in V(G_2)$  must be $r$-step dominated by some $x \in D$. Now $x = u_j$, for some $j$, then $d_{G_1}(v_k, v_j) \leq 1$, contradicting that $v_k$ is not dominated.
    \end{claimproof}
     
For each $r\geq 2$, the \textsc{$r$-Step Domination} Problem is \textsc{W[2]}-hard, even on bipartite graphs.
    \end{proof}

 \begin{restatable}{lemma}{rstepdominationchordal}\label{thm:$r$-step domination chordal} For each $r\geq 2$, the \textsc{$r$-Step Domination} problem is \textsc{W[2]}-hard, even on chordal graphs.
\end{restatable}   
\begin{proof}
    
\begin{figure}[h]
    \centering
   \includegraphics[width=\textwidth]{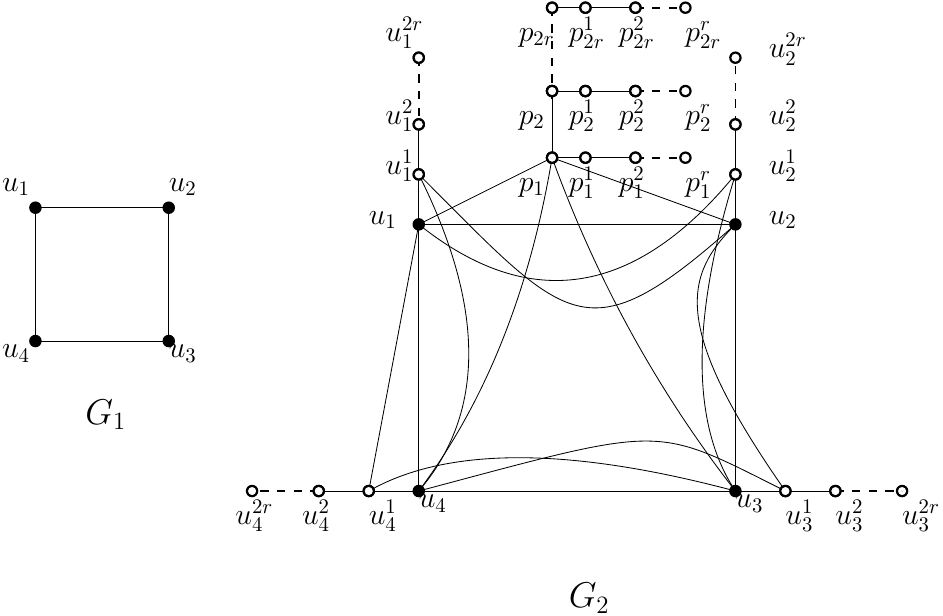}
    \caption{This is a pictorial description of the gadget construction of $G_2$ from $G_1$. Note that, the induced subgraph on the set of vertices $\{ u_i \}$ is a clique.}
    \label{r step hardness}
\end{figure}

Given a graph $G_1$ with a vertex set $V(G_1) = \{v_1,v_2,\dots,v_n\}$, an edge set $E(G_1)$. We construct a new graph $G_2$ in the following way:
(i) We start by taking a copy of \( G_1 \), denoting the vertex set of \( G_2 \) as \( V(G_2) = \{u_1, u_2, \dots, u_n\} \), where each vertex \( u_i \in V(G_2) \) corresponds to \( v_i \in V(G_1) \). 
(ii) For each vertex $u_i$ we add $r$ new vertices $u_i^1,u_i^2,\dots,u_i^r$ and draw a path $u_i-u_i^1-u_i^2-\dots-u_i^r$. (iii) If there is an edge between $v_i$ and $v_j$, i.e $(v_i,v_j)\in E(G_1)$, we draw the edges $(u_i^1,u_j)$ and $(u_i,u_j^1)$. (iv) We add another vertex $p_1$ and draw the edges $(u_i,p_1)$ for all $i=1,2,\dots,n$. (v) We introduce $2r-1$ new vertices $p_2,p_2,\dots,p_{2r}$ and draw the path $p_2-p_3-\dots-p_{(2r-2)}$. (vi) We add a $r$-length path for each $p_i$  by introducing $r$ new vertices $p^1_i,p^2_i,\dots,p^r_i$ and forming the path $p_i-p^1_i-p^2_i-\dots-p^r_i$ and then join $p_i^1$ with $p_i$ by an edge. (vii) Lastly, the
induced subgraph on the set of vertices ${ui}$ is a clique. `

 \begin{claim}\label{claim r step chordal}
           $G_1$ has a dominating set of size at most $k$, if and only if $G_2$ has an $r$-step dominating Set of size at most $k + 2r$, where $r\geq 2$.
    \end{claim}
     
    \begin{claimproof}
        \textsf(If part) Suppose $G_1$ has a dominating set $T$ with $|T| \leq k$. Let $S_1 = \{p_1, p_2,\dots, p_{2r}\}$. This set $r$-step dominates all vertices in $V(G_2) \setminus \{u_i^r \mid i = 1, \dots, n\}$. The remaining vertices $\{u_i^r \mid i = 1, \dots, n\}$ are $r$-step dominated by $T' = \{u_i \mid v_i \in T\}$. Thus, $G_2$ has a $r$-step dominating set of size $\leq k + 2r$.

         \textsf{(Only-if part)} Suppose $G_2$ has a $r$-step dominating set $D$ with $|D| \leq k + 2r$. In order to $r$-step dominate the vertex $p_i^r \in G_2$, $p_i$ must be in $D$. The set $\{p_1,p_2,\dots,p_{2r}\}$ hop dominate all the vertices in $V(G_2)$ except $\{u_1^r,u_2^r,\dots,u_n^r\}$. Thus the vertices $\{u_i^r \mid i = 1, \dots, n\}$ are $r$-step dominated by a set $T' \subseteq D$ with $|T'| \leq k$.
             We can modify $T'$ to contain only vertices from $\{u_1, \dots, u_n\}$. 
             Let $T = \{v_i \mid u_i \in T'\}$. We show that $T$ is a dominating set for $G_1$.

        \medskip

         Suppose there exists $v_k \in V(G_1)$ not dominated by $T$. Then $u_k^r \in V(G_2)$  must be $r$-step dominated by some $x \in D$. Now $x = u_j$, for some $j$, then $d_{G_1}(v_k, v_j) \leq 1$, contradicting that $v_k$ is not dominated. 
    \end{claimproof}
Hence, for each $r\geq 2$, the \textsc{$r$-Step Domination} Problem is \textsc{W[2]}-hard, even on chordal graphs.
\end{proof}

 Unless ETH fails, the Domination problem does not admit an algorithm working in time $2^{o(n+m)}$, where $n$ and $m$ are
the cardinalities of the vertex and edge sets of the input graph, respectively~\cite{cygan2015parameterized}. From Claim \ref{claim r step bipartite} we get the following result:
\begin{corollary}
   Unless ETH fails, the \textsc{$r$-Step Domination} problem does not have $2^{o(n+m)}$ algorithm, even for bipartite graphs and also chordal graphs, where $n $ is the number of vertices and $m$ is the number of edges of the graph.
\end{corollary}
\begin{proof}
    It is known that there is no $2^{o(n)}$-time algorithm for finding a minimum dominating set on $n$ vertices for a general graph, unless ETH fails~\cite{bacso2019subexponential}. Consider the construction that we have used to prove Lemma \ref{thm:$r$-step domination bipartite}. In the proof of Claim \ref{claim r step bipartite}, 
   the graph $G_1$ has $n$ vertex and $m$ edges. We constructed a bipartite graph $G_2$ with $O(n+m)$ vertices and $O(n+m)$ edges.  This reduction proves that a $2^{o(n+m)}$ -time algorithm for finding a minimum $r$-step dominating set on n vertices and m edges for a bipartite graph could be used to obtain a $2^{o(n+m)}$ time algorithm for finding minimum dominating set on general graphs  and this would violate the ETH. This reduction says that there is no $2^{o(n+m)}$-time algorithm for finding minimum $r$-step dominating set on $n$ vertices and $m$ edges for bipartite graphs, unless ETH fails. 
\end{proof}
Minimum \textsc{Dominating Set} problem cannot
be approximated to within a factor of $(1 - \epsilon)\log n$ in polynomial time for any constant $\epsilon > 0$ unless $NP \subseteq DTIME(n^{O(\log \log n})$~\cite{chlebik2008approximation}. Claim \ref{claim r step bipartite} and Claim \ref{claim r step chordal} we get the following result:
\begin{corollary}
For each $r\geq 2$, the  \textsc{$r$-Step Domination} problem cannot be approximated within a factor of $(1- \epsilon)\log n$, even for bipartite graphs and chordal graphs.
\end{corollary}

\section{$r$-Hop Domination (Proof of Theorem~\ref{thm:$r$-hop domination restricted})}\label{sec:rHopDomination}

First we prove that the \textsc{$r$-Hop Domination} problem is in \textsc{W[2]}. To prove that we reduce our problem to the classical \textsc{Dominating Set} problem which is \textsc{W[2]}-complete~\cite{henning2013total}.
\begin{lemma}\label{hop domination w[2]}
    \textsc{$r$-Hop Domination} problem is in \textsc{W[2]}.
\end{lemma}

\begin{proof}
   Let $(G,k)$ be an instance of the \textsc{$r$-Hop Domination} problem. Let $V(G) = \{v_1, v_2, \dots, v_n\}$, we construct a new graph $G'$ as follows:
    

        \begin{enumerate}[]
        \item For each vertex $v_i \in V(G)$, introduce a corresponding vertex $u_i \in V(G')$.
        \item For every pair of vertices $v_j, v_k \in V(G)$ such that $d_G(v_j,v_k)=r$, 
we add the edge $(u_j,u_k)$ to $E(G')$.
    \end{enumerate}
    
    It is now clear that $G$ has an $r$-hop dominating set of size $k$ if and only if $G'$ has a dominating set of size $k$, following the same argument as in Lemma~\ref{r step w[2]}.
\end{proof}
   
To prove that \textsc{$r$-Hop  Domination} Problem is \textsc{W[2]}-hard, even for bipartite graphs and also chordal graphs, we reduce from the \textsc{Domination} problem in general graphs, which is \textsc{W[2]}-hard.

\begin{restatable}{lemma}{rhopdominationbipartite}\label{thm:$r$-hop domination bipartite} For each $r\geq 2$, the \textsc{$r$-Hop Domination} problem is \textsc{W[2]}-hard for bipartite graphs.
\end{restatable}

\begin{proof}

The above same construction used, which is described in Theorem \ref{thm:$r$-step domination bipartite}.
 \begin{claim}\label{claim r hop bipartite}
        $G_1$ has a dominating set of size at most $k$, if and only if $G_2$ has an $r$-hop dominating Set of size at most $k + 2r$, where $r\geq 2$.
    \end{claim}
     
    \begin{claimproof}
        \textsf(If part) Suppose $G_1$ has a dominating set $T$ with $|T| \leq k$. Let $S_1 = \{p_1, p_2,\dots, p_{2r}\}$. This set $r$-hop dominates all vertices in $V(G_2) \setminus \{u_i^r \mid i = 1, \dots, n\}$. The remaining vertices $\{u_i^r \mid i = 1, \dots, n\}$ are $r$-hop dominated by $T' = \{u_i \mid v_i \in T\}$. Thus, $G_2$ has a $r$-hop dominating Set of size $\leq k + 2r$.

         \textsf{(Only-if part)} Suppose $G_2$ has a $r$-hop dominating Set $D$ with $|D| \leq k + 2r$. To $r$-step dominate the vertex Now consider $p_i^r \in V(G_2)$, notice that either $p_i$ or $p_i^r$ must be in $D$. If $p_i^r \in D$ we can modify $D$ by $D\setminus \{p_i^r\}\cup \{p_i\}$ The set $\{p_1,p_2,\dots,p_{2r}\}$ hop dominate all the vertices in $V(G_2)$ except $\{u_1^r,u_2^r,\dots,u_n^r\}$. Thus the vertices $\{u_i^r \mid i = 1, \dots, n\}$ are $r$-hop dominated by a set $T' \subseteq D$ with $|T'| \leq k$.
             We can modify $T'$ to contain only vertices from $\{u_1, \dots, u_n\}$. 
             Let $T = \{v_i \mid u_i \in T'\}$. We show that $T$ is a dominating set for $G_1$.

        \medskip

         Suppose there exists $v_k \in V(G_1)$ not dominated by $T$. Now, $u_k^r \in V(G_2)$,  must be $r$-step dominated by some $x \in D$. Now $x = u_j$, for some $j$, then $d_{G_1}(v_k, v_j) \leq 1$, contradicting that $v_k$ is not dominated.
    \end{claimproof}
    Hence, for each $r\geq 2$, \textsc{$r$-Hop Domination} Problem is \textsc{W[2]}-hard, even on bipartite graphs.
   \end{proof}
    \begin{restatable}{lemma}{rhopdominationchordal}\label{thm:$r$-hop domination chordal} For each $r\geq 2$, the \textsc{$r$-Hop Domination} problem is \textsc{W[2]}-hard for chordal graphs.
\end{restatable}

  \begin{proof}
      
The above same construction used, which is described in Theorem \ref{thm:$r$-step domination chordal}.
 \begin{claim}\label{claim r hop chordal}
         $G_1$ has a dominating set of size at most $k$, if and only if $G_2$ has an $r$-hop dominating Set of size at most $k + 2r$, where $r\geq 2$.
    \end{claim}
     
    \begin{claimproof}
        \textsf(If part) Suppose $G_1$ has a dominating set $T$ with $|T| \leq k$. Let $S_1 = \{p_1, p_2,\dots, p_{2r}\}$. This set $r$-hop dominates all vertices in $V(G_2) \setminus \{u_i^r \mid i = 1, \dots, n\}$. The remaining vertices $\{u_i^r \mid i = 1, \dots, n\}$ are $r$-hop dominated by $T' = \{u_i \mid v_i \in T\}$. Thus, $G_2$ has a $r$-hop dominating Set of size $\leq k + 2r$.

         \textsf{(Only-if part)} Suppose $G_2$ has a $r$-hop dominating Set $D$ with $|D| \leq k + 2r$. To $r$-step dominate the vertex Now consider $p_i^r \in V(G_2)$, notice that either $p_i$ or $p_i^r$ must be in $D$. If $p_i^r \in D$ we can modify $D$ by $D\setminus \{p_i^r\}\cup \{p_i\}$. The set $\{p_1,p_2,\dots,p_{2r}\}$ hop dominate all the vertices in $V(G_2)$ except $\{u_1^r,u_2^r,\dots,u_n^r\}$. Thus the vertices $\{u_i^r \mid i = 1, \dots, n\}$ are $r$-hop dominated by a set $T' \subseteq D$ with $|T'| \leq k$.
             We can modify $T'$ to contain only vertices from $\{u_1, \dots, u_n\}$. 
             Let $T = \{v_i \mid u_i \in T'\}$. We show that $T$ is a dominating set for $G_1$.

        \medskip

         Suppose there exists $v_k \in V(G_1)$ not dominated by $T$. Then $u_k^r \in V(G_2)$,  must be $r$-step dominated by some $x \in D$. Now $x = u_j$, for some $j$, then $d_{G_1}(v_k, v_j) \leq 1$, contradicting that $v_k$ is not dominated.
    \end{claimproof}
    Hence, for each $r\geq 2$, the \textsc{$r$-Hop Domination} Problem is \textsc{W[2]}-hard, even on chordal graphs.
  \end{proof}

     Unless ETH fails, the Domination problem does not admit an algorithm working in time $2^{o(n+m)}$, where $n$ and $m$ are
the sizes of the vertex and edge sets of the input graph, respectively~\cite{cygan2015parameterized}. From Claim \ref{claim r hop bipartite} we get the following result:
\begin{corollary}
   Unless ETH fails, for each $r\geq 2$, the \textsc{$r$-Hop Domination} problem does not have $2^{o(n+m)}$ algorithm, even for bipartite graphs and also chordal graphs, where $n $ is the number of vertices and $m$ is the number of edges of the graph.
\end{corollary}
\begin{proof}
    It is known that there is no $2^{o(n)}$-time algorithm for finding a minimum dominating set on $n$ vertices for a general graph, unless ETH fails~\cite{bacso2019subexponential}. Consider the construction that we have used to prove Lemma \ref{thm:$r$-hop domination bipartite}. In the proof of Claim \ref{claim r hop bipartite}, 
   the graph $G_1$ has $n$ vertex and $m$ edges. We constructed a bipartite graph $G_2$ with $O(n+m)$ vertices and $O(n+m)$ edges.  This reduction proves that a $2^{o(n+m)}$ -time algorithm for finding a minimum $r$-hop dominating set on $n$ vertices and $m$ edges for bipartite graph could be used to obtain a $2^{o(n+m)}$ time algorithm for finding minimum dominating set on general graphs  and this would violate the ETH. This reduction says that there is no $2^{o(n+m)}$-time algorithm for finding a minimum $r$-hop dominating set on $n$ vertices and $m$ edges for bipartite graphs, unless ETH fails. 
\end{proof}
Finding a minimum dominating set cannot
be approximated to within a factor of $(1 - \epsilon)\log n$ in polynomial time for any constant $\epsilon > 0$ unless $NP \subseteq DTIME(n^{O(\log \log n})$~\cite{chlebik2008approximation}. Claim \ref{claim r step bipartite} and Claim \ref{claim r step chordal} we get the following result:
\begin{corollary}
For each $r\geq 2$, the  \textsc{$r$-Hop Domination} problem cannot be approximated to within a factor of $(1- \epsilon)\log n$, even for bipartite graphs and chordal graphs.
\end{corollary}

\section{Conclusion}\label{sec:conclusion}
We studied the parameterized complexity of the exact-distance 
variants \textsc{$r$-Step Domination}, \textsc{$r$-Hop Domination}, and 
\textsc{$r$-Hop Roman Domination} for every $r \ge 2$. 
We proved that \textsc{$r$-Hop Roman Domination} is \textsc{W[2]}-complete, 
and that \textsc{$r$-Step Domination} and \textsc{$r$-Hop Domination} 
are \textsc{W[2]}-complete even on bipartite and chordal graphs. 
Our reductions also yield ETH-based lower bounds ruling out 
$2^{o(n+m)}$-time algorithms, as well as logarithmic inapproximability results.

Together, these results provide a comprehensive complexity classification 
for exact-distance domination variants and show that increasing the distance 
constraint does not make the problems algorithmically easier, even on 
structurally restricted graph classes.

An interesting direction for future work is to investigate whether 
fixed-parameter tractable algorithms exist for these problems on 
more restricted graph classes, or under alternative parameterizations.

\bibliography{bib}
\end{document}